\documentclass[12pt]{amsart}
\usepackage[french,english]{babel}
\usepackage[utf8]{inputenc}
\usepackage{amsmath, amsthm, amsfonts, amssymb}
\usepackage[colorlinks=true, linkcolor=red, citecolor=green]{hyperref}
\usepackage{color}
\usepackage{mathtools}
\usepackage{graphicx}
\usepackage[english, plain]{fancyref}
\usepackage{centernot}
\usepackage{geometry}
\usepackage{enumitem}
\usepackage[makeroom]{cancel}

\usepackage{lipsum}

\newcommand{\newabstract}[1]{%
  \par\bigskip
  \csname otherlanguage*\endcsname{#1}%
  \csname captions#1\endcsname
  \item[\hskip\labelsep\scshape\abstractname.]
}

\newtheorem{proposition}{Proposition}
\newtheorem{theorem}{Theorem}

\newtheorem{ques}{Question}

\begin{document}
\title[A note on a question of Dimca and Greuel]{A note on a question of Dimca and Greuel}
\author[P. Almir\'{o}n]{Patricio Almir\'{o}n}
\author[G. Blanco]{Guillem Blanco}
\thanks{The first author was partially supported by Spanish Ministerio de Ciencia, Innovaci\'{o}n y Universidades MTM2016-76868-C2-1-P. The second author was supported by Spanish Ministerio de Ciencia, Innovaci\'{o}n y Universidades MTM2015-69135-P and Generalitat de Catalunya 2017SGR-932 projects. }

\address{Departamento de \'{A}lgebra, Geometr\'{i}a y Topolog\'{i}a\\
Facultad de Ciencias Matem\'{a}ticas\\
Universidad Complutense de Madrid\\
28040, Madrid, Spain.}
\email{palmiron@ucm.es}
\address{Departament de Matemàtiques\\
Univ. Politècnica de Catalunya\\
Av. Diagonal 647, Barcelona 08028, Spain.}
\email{Guillem.Blanco@upc.edu}

\begin{abstract}
In this note we give a positive answer to a question of Dimca and Greuel about the quotient between the Milnor and Tjurina numbers of an isolated plane curve singularity in the cases of one Puiseux pair and semi-quasi-homogeneous singularities.
\newabstract{french}
Dans cette note, nous donnons une r\'{e}ponse positive \`{a} une question de Dimca et Greuel sur le quotient entre les nombres de Milnor et de Tjurina d'une singularit\'{e} de courbe plane isolée dans le cas d'une paire de Puiseux et de singularit\'{e}s semi-quasi-homog\`{e}nes.
 \end{abstract}
 \selectlanguage{english}
\maketitle
\section{Introduction}
Let \( f : (\mathbb{C}^2, \boldsymbol{0}) \longrightarrow (\mathbb{C}, 0) \) with \( f(\boldsymbol{0}) = 0 \) be a germ of a holomorphic function defining an isolated plane curve singularity. Associated to any isolated plane curve singularity \( f \) one has the Milnor number \( \mu \) and the Tjurina number \( \tau \) that are defined as
\[ \mu := \dim_{\mathbb{C}} \frac{ \mathbb{C} \{x, y\} }{ (\partial f/ \partial x, \partial f/ \partial y) }, \quad \tau := \dim_{\mathbb{C}} \frac{ \mathbb{C} \{x, y\} }{ (f, \partial f/ \partial x, \partial f/ \partial y) }. \]

In \cite{dim}, Dimca and Greuel posed the following question:
\begin{ques} \label{conjecture}
Is it true that $\mu/\tau < 4/3$ for any isolated plane curve singularity?
\end{ques}
Furthermore, they show with an example that this bound is asymptotically sharp.

\vskip 2mm

The purpose of this note is to show that Question \ref{conjecture} has a positive answer, using some known results in two cases: the case of one Puiseux pair and for semi-quasi-homogeneous singularities. By a well-known result of Zariski \cite{zariski-moduli}, the later case contains the former.  However, we decided to include both proofs as the approaches are fundamentally different and may lead to different more general cases of the question. The proof for the first case is based on the results of Delorme \cite{delorme78} and Teissier \cite{teissier-appendix}. For the second case, we use the ideas of Briançon, Granger and Maisonobe \cite{brian}. We also show at the end of this note that the bound also holds for a non-trivial family with two Puiseux pairs studied by Luengo and Pfister \cite{luengo}. All this gives further evidences for a positive answer of the question in the general case.

\vskip 2mm

\textbf{Acknowledgments.} The authors would like to thank, M. Alberich-Carrami\~{n}ana, M. González-Villa, A. Melle-Hern\'{a}ndez and  J. Àlvarez-Montaner for the helpful comments and suggestions.

\section{One Puiseux pair}\label{sec-irred}
In this section we will assume that \( f \) has a single Puiseux pair \( (n, m) \). We will denote by \( \Gamma = \langle n, m \rangle, n < m \) with \( \gcd(n, m) = 1 \) the semigroup of \( f \). Ebey proves in \cite{ebey} that  the moduli space of curves having a given semigroup is in bijection with a constructible algebraic subset of some affine space. For this, he shows that the moduli space is a quotient of an affine space by an algebraic group. Consequently, Zariski \cite[\S VI]{zariski-moduli} defines the generic component of the moduli space as the variety representing the generic orbits of this group action.

\vskip 2mm

Following the ideas of Zariski in \cite{zariski-moduli}, Delorme \cite{delorme78} computed the dimension of the generic component \( q_{n, m} \) of the moduli space of plane branches with a single Puiseux pair \((n, m) \).

\begin{theorem}[{\cite[Thm.~32]{delorme78}}]\label{thmdelorme}
Consider the continued fraction representation \( m/n = [h_1, h_2, \dots, h_k] 
\), with \( k \geq 2, h_1 > 0 \) and \( h_2 > 0 \).
Define, inductively, the following numbers
\[
r_k := 0, \quad t_k := 1, \quad
r_{i-1} := r_i + t_i h_i,  \quad t_{i-1} :=
  \begin{cases}
  \, 0,\quad \textrm{if}\, \ t_i = 1\ \textrm{and}\, \ r_{i-1}\ \textrm{even},\\
  \, 1,\quad \textrm{otherwise}.
  \end{cases}
\]
Then, the dimension \( q_{n,m} \) of the generic component of the moduli space is given by
\[ q_{n, m} = \frac{(n-4)(m-4)}{4} + \frac{r_0}{4} + \frac{(2- t_1)(h_1-2)}{2} - \frac{t_1 t_2}{2}.\]
In particular, except for the case \( (n, m) = (2, 3) \),
\begin{equation} \label{eq:bounds}
\frac{(n-4)(m-4)}{4} \leq q_{n,m} \leq \frac{(n-3)(m-3)}{2}.
\end{equation}
\end{theorem}
The bound in the left-hand side of Equation \ref{eq:bounds} is sharp, consider for instance, the characteristic pair \( n = 8, m = 11 \). In the Appendix \cite{teissier-appendix} of \cite{zariski-moduli}, Teissier, using the monomial curve \( C^\Gamma \), proves that, in general, the dimension \( q \) of the generic component of the moduli space of plane branch with semigroup \( \Gamma \) is given by
\begin{equation} \label{eq:dimension-tau-min}
  q = \tau_{-} - (\mu - \tau_{min}),
\end{equation}
where \( \tau_{-} \) is the dimension of the miniversal constant semigroup deformation of the monomial curve \( C^\Gamma \). For one characteristic exponent we have that \( \tau_{-} \) is the number of points of the standard lattice of \(\mathbb{R}^2\) that are in the interior of the triangle defined by the lines \(\alpha=m-1,\;\;\beta=n-1,\;\;\alpha n+\beta m=nm\), see \cite[\S VI.2]{zariski-moduli}. Therefore, it is easy to see that
\[ \tau_{-} = \frac{(n-3)(m-3)}{2} + \left[\frac{m}{n}\right] - 1, \]
where \( [\, \cdot \,] \) denotes the integer part. In this case, the Milnor number is \( \mu= (n-1)(m-1) \). Combining the lower bound in Equation \ref{eq:bounds} and Equation \ref{eq:dimension-tau-min} one obtains the following lower bound for \( \tau_{min} \)
\begin{equation} \label{eq:bound-tau-min}
\frac{(n-4)(m-4)}{4} + (n-1)(m-1) -\frac{(n-3)(m-3)}{2} - \frac{m}{n} + 1 \leq \tau_{min}.
\end{equation}
except for the case \( (n, m) = (2, 3). \)

\begin{proposition}\label{4/3irreducible}
For any plane branch with one characteristic exponent, \( {\mu}/{\tau} < {4}/{3} \).
\end{proposition}
\begin{proof}

It is sufficient to proof the inequality for the \( \tau_{min} \) of each characteristic pair \( (n, m) \). Dividing \( \mu \) by the expression in Equation \ref{eq:bound-tau-min} and rewriting
\begin{equation} \label{eq:proof-1}
\frac{\mu}{\tau} \leq \frac{\mu}{\tau_{min}} \leq \frac{4 n (n - 1) (m - 1)}{3n^2m - 2n^2 - 2nm + 6n - 4m},
\end{equation}
assuming always that \( (n, m) \neq (2, 3), n < m \). The upper bound in Equation \ref{eq:proof-1} is strictly smaller than \( 4/3 \) if and only if \( 0 < m(n-4) + n(n+3) \). Therefore, the result holds if \( n \geq 4 \). The cases \( n = 2 \) and \( n = 3 \) follow from computing the \( \tau_{min} \) using Theorem \ref{thmdelorme}.

\vskip 2mm

Indeed, let \( n = 2 \) and \( m = 2h_1 + 1, h_1 > 1 \) so the continued fraction representation is \( m/n = [h_1, 2] \). Then, \( r_0 = 2, t_1 = 0, t_2 = 1 \) and \( q_{2, m} = h_1 -m/2 - 1/2 = 0\). Analogously, if \( n =3 \), then \( m = 3h_1 + 1  \) or \( m = 3h_1 + 2 \); the continued fractions are either \( m/n = [h_1, 3] \) or \( m/n = [h_1, 1, 2] \). Then, \( r_0 = 3 + h \) or \( r_0 = 2 + h \), \( t_2 = 1 \) or \( t_2 = 0 \), respectively, and \( t_1 = 1 \) in either case. Consequently, in both cases, \( q_{3, 3h_1+1} = -m/4 + 3h_1/4 + 1/4 = 0 \) and \( q_{3, 3h_1 + 1} = -m/4 + 3h_1/4 + 1/2 = 0 \). Finally, since \( \tau_{-} = 0 \) if \( n= 2 \) and \( \tau_{-} = h_1 - 1 \) if \( n = 3 \),
\[
\frac{\mu}{\tau_{min}} = 1 < \frac{4}{3}, \qquad \frac{\mu}{\tau_{min}} < \frac{6m - 6}{5m-3} < \frac{6}{5} < \frac{4}{3},
\]
for \( n = 2, m \geq 3 \) and \( n = 3, m \geq 4 \), respectively.
\end{proof}

\section{Semi-quasi-homogeneous singularities}\label{sec-multi}
We assume now that \( f \) is a semi-quasi-homogeneous singularity with weights \( w = (n, m) \) such that \( \gcd(n, m) \geq 1 \) and \( n, m \geq 2 \). This means that \( f = f_0 + g \) is a deformation of the initial term \( f_0 = y^n - x^m \) such that \( \deg_w(f_0) < \deg_w(g) \). In \cite{brian}, Brian\c{c}on, Granger and Maisonobe, using the technique of escaliers, give recursive formulas to compute the \( \tau_{min} \) of this type of singularities. Their main result is the following:
\begin{theorem}[{\cite[\S I.6]{brian}}] \label{formulatau}
For semi-quasi-homogeneous singularities with initial term \mbox{\( y^n - x^m \)},
$$\tau_{min}=(m-1)(n-1)-\sigma(m,n).$$
\end{theorem}
The number $\sigma(a,b)$ is defined recursively for any non-negative integers \( a, b \) as follows. If $a, b \leq 2$ then $\sigma(a,b):=0$. Otherwise, we can express $a=bq+r, 0\leq r < b, q\geq 1$. For the cases $r=0,1,b-1,b/2$ there are closed formulas for $\sigma(a,b)$ denoted by $\Sigma_0,\Sigma_1,\Sigma_{b-1}, \Sigma_{b/2}$, see Table 1 in \cite{brian}. If none of the above cases hold, define recursively, see Tables 2 and 3 in \cite{brian}, a finite sequence  $(a_0, b_0), (a_1,b_1), \dots, (a_k, b_k)$ with \( (a_0,b_0) = (m,n) \), $\sigma(a_k,b_k)$ is in one of the previous cases, and for $i=0,\dots,k-1$:
\vskip 1mm
\begin{enumerate}
\item [(A)] If $\gcd(a_i,b_i)=1$, we can find $ub_i-va_i=1$ with $2\leq u < a_i$. Letting $\gamma:=[\frac{a_i-1}{u}]$, we have two subcases:
\vskip 1mm
\begin{enumerate}
\item [(AE)] If $\gamma$ is even, define $a_{i+1}=a_i-\gamma u, b_{i+1}=b_i-\gamma v$, then $$\sigma(a_i,b_i):=\frac{(a_i-2)(b_i-2)}{4}-\frac{(a_{i+1}-2)(b_{i+1}-2)}{4}-\frac{\gamma}{4}+\sigma(a_{i+1},b_{i+1}).$$
\item [(AO)] If $\gamma$ is odd, define $a_{i+1}=(\gamma+1)u-a_i, b_{i+1}=(\gamma+1) v-b_i,$ and $$\sigma(a_i,b_i):=\frac{(a_i-2)(b_i-2)}{4}-\frac{(a_{i+1}-2)(b_{i+1}-2)}{4}-\frac{\gamma+1}{4}+\sigma(a_{i+1},b_{i+1}).$$
\end{enumerate}
\item[(B)] Otherwise, $a_i=\alpha a', b_i=\alpha b'$ with $\alpha\geq 2, \gcd(a',b')=1$, and we can find a Bezout's identity $ub'-va'=1$ with $1\leq u < a'$. We have again two subcases:
\vskip 2mm
\begin{enumerate}
\item [(BP)] If $\alpha$ is even,
$$\sigma(a_i,b_i):=\frac{(a_i-2)(b_i-2)}{4}-\frac{\alpha}{2}.$$
\item [(BO)] If $\alpha$ is odd, define $a_{i+1}=|a'-2u|$ and $b_{i+1}=|b'-2v|$, and
$$\sigma(a_i,b_i):=\frac{(a_i-2)(b_i-2)}{4}-\frac{\alpha}{2}-\frac{(a_{i+1}-2)(b_{i+1}-2)}{4}+\sigma(a_{i+1},b_{i+1}).$$
\end{enumerate}

\end{enumerate}
\begin{proposition}\label{43completo}
For any semi-quasi-homogeneous singularities with initial term $y^n - x^m$, \[ \mu / \tau < 4/3. \]
\end{proposition}
\begin{proof}
Observe that in the recursive cases (A) and (BO),
\vspace{-0.1cm}
\[ \sigma(a, b) \leq \frac{(a-2)(b-2)}{4} - \frac{(a_k - 2)(b_k - 2)}{4} + \sigma(a_k, b_k), \]
where \( \sigma(a_k, b_k) \) is either zero or has a closed form. Notice also that \( a_i b_{i+1} > b_i a_{i+1} \) for all \( i = 0, \dots, k - 1 \). From these observations, one can deduce that, in general,
\vspace{-0.1cm}
\[ (n-1)(m-1) - \frac{(m-2)(n-2)}{4} - \kappa(n, m) \leq \tau_{min}, \]
where \( \kappa(n, m) = m/4n \) if $\sigma(a_k,b_k)$ is $\Sigma_0,\Sigma_1,\Sigma_{b-1}$ with $b$ odd, \( \kappa(n, m) = 5/4 \) if $\sigma(a_k,b_k)$ is $\Sigma_0,\Sigma_1,\Sigma_{b-1}$ with $b$ even or $\Sigma_{b/2}$ with $b/2$ odd, and \(\kappa(n, m) = 0\) if $\sigma(a_k, b_k)$ is $\Sigma_{b/2}$ with $b/2$ even or in the case (BP). In any case,
\[ \frac{\mu}{\tau} \leq \frac{\mu}{\tau_{min}} \leq \frac{4(n-1)(m-1)}{3nm-2n-2m-4 \kappa(n, m)}, \]
which is bounded by \( 4/3 \) if and only if $n + m + \kappa(n, m) > 3$, which is true for \( n, m \geq 2 \).
\end{proof}

\vspace{-0.2cm}
\section{A family with two Puiseux pairs}

In \cite{luengo}, Luengo and Pfister study the family of irreducible plane curve singularities with semigroup \( \langle 2p, 2q, 2pq + d \rangle \) such that \( \gcd(p, q) = 1, p < q \) and \( d \) odd. The Milnor number of this family equals
\vspace{-0.1cm}
\[ \mu = (2p - 1)(2q - 1) + d. \]
Studying the kernel of the Kodaira-Spencer map, they prove, see \cite[pg. 259]{luengo}, that \( \tau \) is constant in each equisingularity class and equals,
\vspace{-0.1cm}
\[ \tau = \mu - (p-1)(q-1). \]
One can easily check that \( \mu / \tau < 4/3 \) for all the semigroups of the family.

\end{document}